\documentclass[12pt]{amsart}
\usepackage[usenames]{color} 
\usepackage{times}
\usepackage{amsfonts}
\usepackage{amsthm}
\usepackage{amsmath}
\usepackage{psfrag}
\usepackage{times}
\usepackage{latexsym,color}
\usepackage{amssymb}
\usepackage{graphicx}

\usepackage{epsfig}
\usepackage{gastex}

\advance \topmargin by -\headheight
\advance \topmargin by -\headsep
\evensidemargin \oddsidemargin
\usepackage{amsfonts, amssymb, amsmath, amsthm}
\newcommand{\R}{\mathbb{R}}
\newcommand{\N}{\mathbb{N}}
\newcommand{\Q}{\mathbb{Q}}
\newcommand{\Z}{\mathbb{Z}}

\newcommand{\T}{\mathbb{T}}

\newcommand{\Cech}{{\v Cech}{} }
\newcommand{\OP}{\Omega_{\Phi}}

\newcommand{\larr}{\left( \begin{array}{c}}
\newcommand{\rarr}{\end{array} \right) }

\newcommand{\lsqarr}{\left[ \begin{array}{c}}
\newcommand{\rsqarr}{\end{array} \right]}

\newcommand{\arrow}{\rightarrow}

\newcommand{\inv}{\varprojlim}
\newcommand{\dir}{\varinjlim}
\def\coker{\mathop{\rm coker}\nolimits}
\def\rank{\mathop{\rm rank}\nolimits}

\newtheorem{theorem}{Theorem}
\newtheorem{Theorem}[theorem]{Theorem}

\newtheorem{corollary}[theorem]{Corollary}
\newtheorem{conjecture}[theorem]{Conjecture}
\newtheorem{lemma}[theorem]{Lemma}
\newtheorem{prop}[theorem]{Proposition}
\newtheorem{example}[theorem]{Example}

\numberwithin{equation}{section}

\begin{document}

\title[Factors of Pisot tiling spaces]{Factors of Pisot tillng spaces and the Coincidence Rank Conjecture}

\author[M. Barge]{Marcy Barge}
\address{Department of Mathematics\\
Montana State University\\
Bozeman, MT 59717-0240, USA.}
\email{barge@math.montana.edu}

\keywords{Pisot substitution, tiling space}
\subjclass[2000]{Primary 68R15 \& 05D10}
\date{October 14, 2012}

\maketitle

\begin{abstract} We consider the structure of Pisot substitution tiling spaces,
in particular, the structure of those spaces for which the translation action does not have pure discrete spectrum. Such a space is always a measurable $m$-to-one cover of an
action by translation on a group called the maximal equicontinuous factor. The integer $m$ is the coincidence rank of the substitution and equals one if and only if translation on the tiling space has pure discrete spectrum. By considering factors intermediate between a tiling space and its maximal equicontinuous factor, we establish a lower bound on the cohomology of a one-dimensional Pisot substitution tillng space with coincidence rank two and dilation of odd norm. The Coincidence Rank Conjecture, for coincidence rank two,
is a corollary.
 
\end{abstract}

\section{Introduction}

A subset $\mathcal{D}$ of $\R^n$ is a {\em Delone set} if it is relatively dense and uniformly discrete. Such a set has {\em finite local complexity} (FLC) if, for each $r>0$, there are, up to translation,  only finitely many subsets of $\mathcal{D}$ of diameter less than $r$ and is  {\em repetitive} if, for each $r$ there is an $R$ so that every Euclidean ball of radius $R$ intersected with $\mathcal{D}$ contains a translated copy of every subset of $\mathcal{D}$ of diameter less than $r$. Repetitive FLC Delone sets (RFLC) are commonly employed as models of the atomic structure of materials and there is a well-developed diffraction theory to go with these models (\cite{BMRS},\cite{L}, \cite{Le}). Of course if a Delone set is a periodic lattice, then its diffraction spectrum is pure point. But there is also a vast menagerie of non-periodic repetitive
FLC Delone sets with pure point diffraction spectra (the so-called quasicrystals) 
and it is a very subtle problem to predict which sets will have this property.

An effective method of encoding the structure of a single Delone set $\mathcal{D}$ is to 
form the space, called the {\em hull} of $\mathcal{D}$ and denoted $\Omega_{\mathcal{D}}$, consisting of all Delone sets that are, up to translation, locally indistinguishable from $\mathcal{D}$. The hull has a metric topology in which two Delone sets are close if a small translate of one is identical to the other in a large neighborhood of the origin. If $\mathcal{D}$ is RFLC then $\Omega_{\mathcal{D}}$ is compact and connected and $\R^n$ acts minimally on $\Omega_{\mathcal{D}}$ by translation. A highlight of this approach is that (in the substitutive context considered here) $\mathcal{D}$ has pure point diffraction spectrum if and only if the $\R^n$-action on $\Omega_{\mathcal{D}}$ has pure discrete 
dynamical spectrum (\cite{D},\cite{LMS}). 

In this article we consider Delone sets that arise from a substitution rule. In this case, the 
eigenfunctions of the $\R^n$-action on $\Omega_{\mathcal{D}}$ can all be taken continuous (\cite{S}) and then collectively determine a continuous map $g$ factoring the $\R^n$ action on $\Omega_{\mathcal{D}}$ to a translation action on a compact abelian group. The map $g$ is called the {\em maximal equicontinuous factor map} and 
it turns out that the $\R^n$-action on $\Omega_{\mathcal{D}}$ has pure discrete spectrum if and only if $g$ is a.e. one-to-one (\cite{BK}).

Two ingredients make up a substitution: a linear expansion $\Lambda:\R^n\to\R^n$; and a subdivision rule. For the $\R^n$-action on $\Omega_{\mathcal{D}}$ to have any discrete component in its spectrum, the eigenvalues 
of $\Lambda$ must satisfy a a {\em Pisot family} condition (\cite{LS}). In particular, if $\Lambda$ is simply a dilation by the number $\lambda>1$ (the ``self-similar" case), the $\R^n$-action will have a discrete component if and only if $\lambda$ is a Pisot number. (Recall that a Pisot number is an algebraic integer, all of whose algebraic conjugates lie strictly inside the unit circle.) For the purpose of determining which substitution Delone sets $\mathcal{D}$ have pure point diffraction spectrum, we are thus led to consider factors of the spaces $\Omega_{\mathcal{D}}$ for which the expansion is Pisot family.

For convenience, we will consider substitution tilings rather than substitution Delone sets.
(These notions are essentially equivalent - see \cite{LMS} and \cite{LW}.) A {\em tile} is a compact, topologically regular subset of $\R^n$ (perhaps ``marked"); for a given substitution, there are only finitely many translation equivalence classes of tiles and each such class is a {\em type}. A substitution expands tiles by a linear map $\Lambda$, then replaces the expanded tile by a collection of tiles. The {\em substitution matrix} is the $d\times d$ matrix $M$ ($d$ being the number of distinct tile types) whose $ij-$th entry is the number of tiles of type $i$ replacing an inflated tile of type $j$ and the {\em tiling space} associated with a substitution $\Phi$ is the collection $\OP$ of all tilings $T$ of $\R^n$ with the property that each finite patch of $T$ occurs as a sub patch of some repeatedly inflated and substituted tile. To pass from a tiling $T$ to a Delone set $\mathcal{D}$, simply pick a point from each tile in $T$: if the points picked are ``control points" (see \cite{LS}), $\mathcal{D}$ will be a substitution Delone set with linear expansion $\Lambda$. For a thorough 
treatment of the basics of substitution tiling spaces, see \cite{AP}.

A 1-dimensional Pisot substitution is {\em irreducible} if the characteristic polynomial of its substitution matrix is irreducible over $\Q$, and is {\em unimodular} if that matrix has determinant $\pm1$. The following has become known as the {\bf Pisot Substitution Conjecture}:

\begin{conjecture}\label{PC} (PSC) If $\Phi$ is a 1-dimensional irreducible, unimodular, Pisot substitution  then the $\R$-action on $\OP$ has pure discrete spectrum.

\end{conjecture}

The dimension of the first ({\Cech}, with rational coefficients) cohomology of a 1-dimensional substitution tiling space is at least as large as the degree of the associated dilation. Replacing irreducibility in the
PSC by minimality of cohomology results in the {\bf Homological Pisot Conjecture}:

\begin{conjecture}\label{HPC} (HPC) If $\Phi$ is a 1-dimensional substitution with dilation a Pisot unit of degree $d$ and the first {\Cech} cohomology $H^1(\OP)$ has dimension $d$,
then the $\R$-action on $\OP$ has pure discrete spectrum.
\end{conjecture} 

As noted above for substitution Delone sets, the $\R^n$-action on $\OP$ has pure discrete spectrum if and only if the maximal equicontinuous factor map $g$ is a.e. one-to-one. It is proved in \cite{BK} that, for Pisot family $\Phi$, there is $cr(\Phi)\in\N$ 
(called the {\em coincidence rank} of $\Phi$) so that $g$ is a.e. $cr(\Phi)$-to-one. The following {\bf Coincidence Rank Conjecture} ( see \cite{BBJS}) extends the HPC to the non-unit case.

\begin{conjecture}\label{CRC} (CRC) If $\Phi$ is a 1-dimensional substitution with Pisot
dilation $\lambda$ and the dimension of $H^1(\OP)$ equals the degree of $\lambda$, then $cr(\Phi)$ divides the norm of $\lambda$.
\end{conjecture} 

For discussions of the various Pisot conjectures, and extensions to higher dimensions, see \cite{BS}, \cite{BBJS}, \cite{BG}, and, particularly, \cite{ABBLSS}.

Conjecture \ref{CRC} is established in \cite{BBJS} for $\lambda$ of degree 1. The main result of this article verifies Conjecture \ref{CRC} for $\lambda$ of any degree in case the coincidence rank is 2. We prove (Theorem \ref{big H^1} in the text):
\\
\\
\noindent {\bf Theorem}: Suppose that  $\Phi$ is a 1-dimensional substitution with Pisot dilation $\lambda$ of degree $d$. If the norm of $\lambda$ is odd and $cr(\Phi)=2$, then dim$(H^1(\OP))\ge2d-1$.
\\
\\

To prove the above, we will show that the maximal equicontinuous factor map
$$\OP\stackrel{g}\arrow\T_{max}$$
for a Pisot family tiling space $\OP$ with coincidence rank 2 (and with no almost automorphic sub actions) factors as
$$\OP\stackrel{\pi_s}\arrow\Omega_s\stackrel{\pi_p}\arrow\Omega_p\stackrel{g'}\arrow\T_{max}$$
with $\Omega_s$ and $\Omega_p$ also substitution tiling spaces, $\pi_s$ and $g'$ measure isomorphisms, and $\pi_p$ a topological 2-to-1 covering projection. The space 
$\Omega_p$ is the {\em maximal pure discrete factor} of $\Omega_s$ - examples have been considered in \cite{BGG} where its utility in analyzing the non-discrete part of the spectrum of $\OP$ is explored. Restricting to dimension 1, the map $\pi_s$ induces an injection in cohomology and the lower bound in the above theorem follows from arithmetic associated with the double cover. The number two seems to be quite special in this argument; it would be extremely interesting to see an extension to higher coincidence rank. 

\section{Preliminaries}

In this article, all substitutions will be primitive (some power of the substitution matrix is strictly positive), non-periodic (no element of the tiling space has a non-zero translational period), and FLC. If $\Phi$ is an $n$-dimensional substitution, the tiling space $\OP$ with the tiling metric is then compact, connected, and locally the product of a Cantor set with an $n$-dimensional disk. The action of $\R^n$ on $\OP$ is minimal and uniquely ergodic, the substitution induced map $\Phi:\OP\to\OP$ is a homeomorphism, and $\Phi(T-v)=\Phi(T)-\Lambda v$ for all $T\in\OP$ and $v\in\R^n$, $\Lambda$ being the linear inflation associated with $\Phi$. 

If $\phi:\mathcal{A}=\{1,\ldots,m\}\to\mathcal{A}^*$ is a symbolic substitution with matrix $M$ (the $ij$-th entry is the number of $i$'s occurring in $\phi(j)$), there is an associated $1$-dimensional tile substitution $\Phi$ on a set of $m$ prototiles: the $i$-th prototile is the interval $[0,\omega_i]$, marked with the symbol $i$ (formally, the $i$-th prototile is the pair $([0,\omega_i],i)$), where $\omega=(\omega_1,\ldots,\omega_m)$ is a positive left eigenvector for $M$. The linear expansion for $\Phi$ is the Perron-Frobenius eigenvalue $\lambda$ of $M$: $\omega M=\lambda\omega$.

We will say that the substitution tiling spaces $\OP$ and $\Omega_{\Psi}$ are {\em isomorphic}, and write $\OP\simeq\Omega_{\Psi}$, if there is a homeomorphism of $\OP$ with $\Omega_{\Psi}$ that conjugates both the $\R^n$-actions and the substitution induced homeomorphisms.
The {\em stable set} of $T\in\OP$ is the set $W^s(T)=\{T'\in\OP:d(\Phi^n(T')\Phi^n(T))\to0\text{ as }n\to\infty\}$. We write $T\sim_s T'$ if $T'\in W^s(T)$ and we say that $T$ and $T'$ are {\em stably equivalent}.

The eigenvalues of the linear expansion $\Lambda$ of a substitution are necessarily algebraic integers. Given $k\in\N$ and $\alpha\in\mathbb{C}$, let $n(k,\alpha)$ denote the number of $k\times k$ Jordan blocks in the Jordan form of $\Lambda$ that have eigenvalue $\alpha$. Following \cite{LS} and \cite{K}, we will say that $\Phi$ is a {\em Pisot family substitution} if the eigenvalues, 
$spec(\Lambda)$, of $\Lambda$ satisfy the condition: if $\lambda\in spec(\Lambda)$ and $\eta$ is an algebraic conjugate of $\lambda$ with $|\eta|\ge1$, then $n(k,\eta)= n(k,\lambda)$ for all $k\in\N$.  A substitution is {\em self-similar} if its expansion is a scalar, $\Lambda=\lambda I$: such a substitution is Pisot family if and only if $\lambda$ is a Pisot number  and we will call such substitutions simply {\em Pisot substitutions}.

Rather generally, group actions on compact spaces have maximal equicontinuous factors. For the $\R^n$-action on a substitution tiling space $\OP$, the maximal equicontinuous factor is 
a Kronecker action on a solenoidal group (an inverse limit of linear hyperbolic endomorphisms of tori): $$g:\OP\to \T_{max}.$$
The ``maximality" property of $g$ is that if $f:\OP\to X$ is any map that factors the $\R^n$-action on $\OP$ onto an equicontinuous $\R^n$-action on $X$, then $f$ factors through $g$: there is a continuous factor map $h:\T_{max}\to X$ with $f=h\circ g$. 
If $\Phi$ is an $n$-dimensional Pisot substitution with expansion $\lambda$ of algebraic degree $d$, $\T_{max}$
has dimension $nd$. If $\lambda$ is a unit, then $\T_{max}$ is just the $nd$-torus.
For general abelian group actions on compact spaces, the equicontinuous structure relation (the quotient by which gives the maximal equicontinuous factor) is the regional proximal relation. For Pisot family substitution tiling spaces, this relation takes the stronger form given by (6) of Theorem \ref{tools} below.
\\

The following theorem \footnote[1]{The definition of Pisot family used in \cite{BK}, where the results on which Theorem \ref{tools} is based are established, is narrower than that used in the present paper. By \cite{K} those results extend to the context here.} collects results that we will require later. We denote by $B_r[T]$ the collection of all tiles of $T$ whose supports meet the closed ball of radius $r$ centered at 0 and we define the {\em coincidence rank} of $\Phi$ to be $$cr(\Phi)=min\{\sharp g^{-1}(z):z\in\T_{max}\}.$$
(In \cite{BK}, this is called the {\em minimal rank} of  $g:\OP\to\T_{max}$. In \cite{BKw},
$cr(\Phi)$ is the maximum cardinality of a collection of tilings in the same fiber of $g$, no two of which share a tile (i.e., coincide). The minimal rank and coincidence rank are shown to be the same for Pisot family substitutions in \cite{BK} - see (2) of Theorem \ref{tools} below.) The ``a.e."  in the theorem refers to Haar measure on $\T_{max}$.
\\
\\
\begin{theorem}\label{tools} Let $\Phi$ be a Pisot family substitution and let $g:\OP\to\T_{max}$ be the map onto the maximal equicontinuous factor. Then:
\begin{enumerate}
\item g is finite-to-one and a.e. $cr(\Phi)$-to-one.
\item For each $z\in\T_{max}$ there are $T_1,\ldots,T_{cr(\Phi)}\in g^{-1}(z)$ so that 
$\Phi^k(T_i)\cap \Phi^k(T_j)=\emptyset$ for $i\ne j$, and all $k\in\Z$, and if $T\in g^{-1}(z)$ then $T\cap T_i\ne\emptyset$ for some $i$.
\item $T\sim_s T'$ if and only if there is $k\in\N$ with $B_0[\Phi^k(T)]=B_0[\Phi^k(T')]$.
\item Up to translation, there are only finitely many pairs of the form $(B_0[T]),B_0[T'])$ with $g(T)=g(T')$.

\item The $\R^n$-action on $\OP$ has pure discrete spectrum $\Longleftrightarrow$ $g$ is a.e one-to-one $\Longleftrightarrow$ $cr(\Phi)=1$.
\item $g(T)=g(T')$ if and only if for every $r>0$ there are $S_r,S'_r\in\OP$ and $v_r\in\R^n$ so that $B_r[T]=B_r[S_r]$, $B_r[T']=B_r[S'_r]$, and $B_r[S_r-v_r]=B_r[S'_r-v_r]$.
\item If $T-v\sim_s T'-v$ for all $v\in\R^n$, then $T=T'$.
\end{enumerate}
\end{theorem}
\begin{proof}
(1) $g$ is finite-to-one by Theorem 5.3 of \cite{BK} and a.e. $cr(\Phi)$-to-one by the proof of Theorem 2.25 of \cite{BK}.

(2) From Theorem 5.4 and Lemma 2.14 of \cite{BK}, there is $\delta>0$ so that, for each $z\in\T_{max}$, there are $T_1,\ldots,T_{cr(\Phi)}\in g^{-1}(z)$ with $\inf_{v\in\R^n}d(T_i-v,T_j-v)\ge\delta$ for $i\ne j$. Moreover, if $T\in g^{-1}(z)$, then $\inf_{v\in\R^n}d(T-v,T_i-v)=0$ for some $i\in\{1,\ldots,cr(\phi)\}$. If, for every such $\{T_1,\ldots,T_{cr(\Phi)}\}$, there are $i\ne j$ so that $\Phi^k(T_i)\cap\Phi^k(T_j)\ne\emptyset$ for some $k$, then, for sufficiently large $k$ there would not be $T'_1,\ldots,T'_{cr(\Phi)}\in g^{-1}(z')=\Phi^k(g^{-1}(z))$ with $\inf_{v\in\R^n}d(T'_i-v,T'_j-v)\ge\delta$, for $i\ne j$, as required.
Furthermore, if $T\in g^{-1}(z)$ were disjoint from all the $T_i$, then $\inf_{v\in\R^n}d(T-v,T_i-v)>0$, by (4) of this theorem, and then, by minimality of the $\R^n$-action, $g$ would be at least $(cr(\Phi)+1)$-to-one everywhere, in contradiction to (1).

(3) See Lemma 3.6 of \cite{BO}

(4) See Corollary 5.8 of \cite{BK}.

(5) See \cite{BK}.

(6) The condition that for every $r>0$ there are $S_r,S'_r\in\OP$ and $v_r\in\R^n$ so that $B_r[T]=B_r[S_r]$, $B_r[T']=B_r[S'_r]$, and $B_r[S_r-v_r]=B_r[S'_r-v_r]$ is called {\em strong regional proximality} in \cite{BK}: that this is equivalent to $g(T)=g(T')$ follows from Theorems 5.6 and 3.4 of that paper.

(7) See Lemma 3.10 of \cite{BO}.
\end{proof}

\section{Pisot Factors}

 We will write $T\approx_s T'$, provided $\{v:T-v\sim_sT'-v\}$ is dense in $\R^n$ and $g(T)=g(T')$. It follows from (3) of Theorem \ref{tools} that $\{v:T-v\sim_s T'-v\}$ is open and, from this, that $\approx_s$ is an equivalence relation. Since  whether or not $T\sim_s T'$ is entirely dependent on 
$(B_0[T],B_0[T'])$, it follows from (4) of Theorem \ref{tools} that the relation $\approx_s$ is also closed. The cardinalities of the $\approx_s$-equivalence classes are uniformly bounded by $cr(\Phi)$. It is clear that $T\approx_s T'\Longleftrightarrow \Phi(T)\approx_s\Phi(T')$ and $T-v\approx_s T'-v$ for all $v\in\R^n$. Thus the $\Z$- and $\R^n$-actions on $\OP$ induce $\Z$- and $\R^n$-actions on the quotient $\OP/\approx_s$. 

\begin{Theorem}\label{pi=g} $\OP/\approx_s$ is the maximal equicontinuous factor of the $\R^n$-action on $\OP$ iff the $\R^n$-action on $\OP$ has pure discrete spectrum.
\end{Theorem} 
\begin{proof} 
Suppose that the $\R^n$-action on $\OP$ has pure discrete spectrum and suppose that $g(T)=g(T')$. If there are $v_0$ and $\epsilon>0$ so that $T-v_0-v\nsim_sT'-v_0-v$ for each $v\in B_{\epsilon}(v_0)$, choose $n_i\to\infty$ such that $\Phi^{n_i}(T-v_0)\to S\in\OP$ and $\Phi^{n_i}(T'-v_0)\to S'\in\OP$. Then $g(S)=g(S')$ and $S\cap S'=\emptyset$. But this means that the coincidence rank of $\Phi$ is at least 2, so the spectrum of the $\R^n$-action is not pure discrete ((5) of Theorem \ref{tools}). Thus there are no such
$v_0,\epsilon$ and $T-v\sim_s T'-v$ for a dense set of $v$. Thus $g(T)=g(T')\implies
T\approx_s T'$. That is, $\approx_s$ is the equicontinuous structure relation and $\OP/\approx_s$ is the maximal equicontinuous factor.

Conversely, suppose that the spectrum of the $\R^n$-action on $\OP$ does not have pure discrete spectrum. Then $cr(\Phi)\ge2$ and, in particular, there are tilings $T,T'\in\OP$ that are periodic under $\Phi$ with $g(T)=g(T')$ and $T\cap T'=\emptyset$. But then $\Phi^k(T)\cap\Phi^k(T')=\emptyset$ for all $k\in\N$ so that $T-v\nsim_s T'-v$ for {\em any} $v$. Thus it is not the case that $g(T)=g(T')\implies T\approx_sT'$, so $\approx_s$ is not the equicontinuous structure relation and $\OP/\approx_s$ is not the maximal equicontinuous factor.

\end{proof}

We will see that if $\Phi$ is a 1-dimensional Pisot substitution, and the $\R$-action on $\OP$ does not have pure discrete spectrum, then $\OP/\approx_s$ is also a 1-dimensional Pisot substitution tiling space. The most straightforward generalization of this is not true in higher dimensions as the following simple example shows.

\begin{example} \label{products} Let $\psi$ be the Thue-Morse substitution $$\psi:a\mapsto ab,\, b\mapsto ba$$
and let $\phi$ be the Fibonacci substitution $$\phi: a\mapsto ab,\,b\mapsto a.$$
The corresponding tile substitutions, $\Psi$ and $\Phi$, are Pisot and the product
$\Psi\times\Phi$ defines a 2-dimensional Pisot family substitution on rectangular tiles.
The relation $\approx_s$ is trivial on $\Omega_{\Psi}$ while $\OP/\approx_s$ is the 2-torus. It is not hard to see that $\Omega_{\Psi\times\Phi}/\approx_s$ is homeomorphic with $\Omega_{\Psi}\times\T^2$, which is not a 2-dimensional substitution tiling space.

\end{example}

In the above example, even though the $\R^2$-action on the tillng space does not have pure discrete spectrum, there are minimal sets under 1-dimensional sub-actions that do have pure discrete spectrum. A difficulty in trying to capture this sort of phenomenon in a general setting is that arbitrary sub actions might not carry unique invariant measures, making it difficult to meaningfully speak of their dynamical spectra. Rather, we can measure how close a sub action is to being equicontinuous.

The {\em minimal rank} of a minimal $\R^k$-action on a space $X$ is the the minimum cardinality of a fiber $g_X^{-1}(g(x))$ of the maximal equicontinuous factor map $g_X:X\to X_{max}$ for the action. For the full $\R^n$-action on an $n$-dimensional Pisot family substitution tiling space the minimal rank is the same as the coincidence rank and hence equals 1 if and only if the spectrum is pure discrete (Theorem \ref{tools}). Let us say that the $n$-dimensional tiling space 
$\Omega$ 
 
{\em has an almost automorphic sub action}
if there is a nontrivial subspace $V\subset\R^n$ and a minimal set $X\subset\Omega$ under the $V$-action $\Omega\times V\ni(T,v)\mapsto T-v$ on which this $V$-action has 
 minimal rank 1.
In Example \ref{products}, the $V$-action determined by $V=\{(0,t):t\in\R\}$ on $X:=\{T\}\times\Omega_{\Phi}$ has minimal rank 1 for every $T\in\Omega_{\Psi}$. We will prove
(Theorem \ref{minimal model}), for general Pisot family $\Phi$, that $\OP/\approx_s$ is a substitution tiling space if and only if $\OP$ has no almost automorphic sub actions.

Recall that a substitution $\Phi$ {\em forces its border} provided there is $m\in\N$ so that
if $0\in\tau\in T\cap T'$ for any $T,T'\in\OP$, then $B_0[\Phi^m(T)]=B_0[\Phi^m(T')]$. By, for example, introducing collared tiles, we may always replace a substitution by one that forces its border and has an isomorphic tiling space (see \cite{AP}).  
The {\em Anderson-Putnam complex}, $X_{\Phi}$, for $\Phi$ is an $n$-dimensional CW-complex whose $n$-cells are the prototiles with prototiles glued along an $n-1$ face
if they have representative tiles occurring in a tiling $T\in\OP$ and intersecting along that face. More precisely, $X_{\Phi}=\OP/\sim$ with $\sim$ the smallest equivalence relation
for which $0\in\tau\in T\cap T'\implies T\sim T'$. Let $\pi:\OP\to X_{\Phi}$ be the quotient map. There is then a map $f_{\Phi}:X_{\Phi}\to X_{\Phi}$ so that $f_{\Phi}\circ\pi=\pi\circ\Phi$ and hence a map $\hat{\pi}:\OP\to\inv f_{\Phi}$ that semi-conjugates $\Phi$ with the shift homeomorphism $\hat{f}_{\Phi}:\inv f_{\Phi}\to\inv f_{\Phi}$. If $\Phi$ forces its border, $\hat{\pi}$ is a homeomorphism (\cite{AP}).

We wish to transfer the relation $\approx_s$ to $X_{\Phi}$. Let  $R$ be the relation on $X_{\Phi}$ defined by $$xRx'\Longleftrightarrow \exists\, T,T'\in\OP\text{ with }T\approx_s T',\, \pi(T)=x, \text{ and }\pi(T')=x'.$$
 The relation $R$ is clearly reflexive and symmetric and is also topologically closed (see the proof of Lemma \ref{equiv rel}), but it may not be transitive. And the smallest equivalence
relation containing $R$ may not be closed (for the Fibonacci substitution, all equivalence classes are dense). 

Given patches (collections of tiles) $P$ and $P'$ and $x\in\R^n$, let us write $P-x\sim_sP'-x$ provided $x$ is in the interior of the intersection of the supports of $P$ and $P'$ and $T-x\sim_s T'-x$ for all (equivalently, some) tilings $T,T'$ with $P\subset T$ and $P'\subset T'$. If $P-x\sim_s P'-s$ for a set of $x$ dense in the interior of the intersection of the supports of $P$ and $P'$, we'll say that $P$ and $P'$ are {\em densely stably equivalent on overlap}. 

\begin{theorem} \label{PDS} Suppose that $\Phi$ is an $n$-dimensional substitution of Pisot family type and that there are $0\ne v_k\in\R^n$ with $v_k\to0$ and a tile $\tau$ so that 
$\tau$ and $\tau-v_k$ are densely stably equivalent on overlap for each $k$. Then $\OP$ has an almost automorphic sub action.
\end{theorem} 

The proof of Theorem \ref{PDS} will require some buildup. Given a nontrivial subspace 
 $V$ of $\R^n$, let's say that $\emptyset\ne X\subset\OP$ is {\em $V$-minimal} if $X$ is a minimal set for the $V$-action. 

\begin{lemma}\label{full meas} If $\Phi$ is an $n$-dimensional substitution of Pisot family type and $V$ is a nontrivial subspace of $\R^n$, then the union of all $V$-minimal sets, $Y:=\cup_{X\text{ V-minimal}}X$, has full measure with respect to the unique ergodic (under the $\R^n$-action) measure on $\OP$.
\end{lemma}

\begin{proof} Let $(z,v)\mapsto z-v$ denote the $\R^n$-action on $\T_{max}$, so that
$g(T-v)=g(T)-v$ for all $T\in\OP$, $v\in\R^n$. For each $z\in\T_{max}$, the set $\omega(z):=cl\{z-v:v\in V\}$ is $V$-minimal. The set $g^{-1}(\omega(z))$ is closed and $V$-invariant, hence contains a $V$-minimal set, call it $X_z$. It must be the case that $g(X_z)=\omega(z)$. Thus $g(Y)\subset g(\cup_{z\in\T_{max}}X_z)=\T_{max}$. Let $G=\{z\in\T_{max}:\sharp g^{-1}(z)=cr(\Phi)\}$. $G$ has full measure in $\T_{max}$ by (1) of Theorem \ref{tools}. It is proved in \cite{BK} (see the proof of Theorem 2.25 there) that
$G':=g^{-1}(G)$ has full measure in $\OP$ and
$g|_{G'}:G'\to G$ is a topological $cr(\Phi)$-to-1 covering map. Furthermore, if $U$ is a (nonempty) relatively open set of $G$ evenly covered by the relatively open sets $U_1,\ldots,U_{cr(\Phi)}$ in $G'$, then $g|_{U_i}$ is a measure isomorphism of normalized measures for each $i$. Since $g(Y\cap g^{-1}(U))=U$ and $U$ has positive measure, it follows that 
$Y$ has positive measure. Note that if $X$ is a $V$-minimal set and $w\in\R^n$, then $X-w$ is also $V$-minimal. Thus $Y$ is $\R^n$-invariant and by ergodicity must have full measure.
\end{proof}

Given a $V$-minimal set $X$ in $\OP$, let $g_X:X\to X_{max}$ denote the maximal equicontinuous factor (with respect to the $V$-action). Since $g|_X:X\to g(X)\subset\T_{max}$ is also an equicontinuous factor, $g|_X$ factors through $g_X$: there is $p_X:X_{max}\to\T_{max}$ with $g|_X=p_X\circ g_X$. That is, if $T,T'\in X$ are regionally proximal with respect to the $V$-action, then $T$ and $T'$ are regionally proximal with respect to the $\R^n$-action.

\begin{lemma}\label{rpV} Suppose that $\Phi$ is an $n$-dimensional substitution of Pisot family type, $V$ is a nontrivial subspace of $\R^n$, and $X\subset \OP$ is a $V$-minimal set. If $T,T'\in X$ are regionally proximal with respect to the $V$-action, then there are $S,S'\in X$ and $v\in V$ so that $B_0[T]=B_0[S]$, $B_0[T']=B_0[S']$, and $B_0[S-v]=B_0[S'-v]$.
\end{lemma}
\begin{proof} Let $\rho_1,\ldots,\rho_m$ be the prototiles for $\Phi$ and choose $c_i\in spt(\rho_i)$ for each $i$. Let $\Xi:=\{v=(y_i-c_i)-(y_j-c_j): \text{ there is } T\in\OP\text{ with } \rho_i+y_i,\rho_j+y_j\in T\}$. Then $\Xi$ is a Meyer set (\cite{LS}). That is, $\Xi-\Xi$ is uniformly discrete, as is $\Xi\pm\Xi\pm\cdots\pm\Xi$ for any choice of signs (\cite{M}).
Since $T$ and $T'$ are also regionally proximal with respect to the full $\R^n$-action,
there are $P,P'\in\OP$ and $w\in\R^n$ so that $B_0[T]=B_0[P]$, $B_0[T']=B_0[P']$,
and $B_0[P-w]=B_0[P'-w]$ ((6) of Theorem \ref{tools}). Suppose that $\rho_i+y_i$, $\rho_j+y_j$, and $\rho_k+y_k$ are tiles in $T,T'$, resp. $S-w$, with support containing 0.
Then $((y_j+c_j)-(y_k+c_k))-((y_i+c_i)-(y_k+c_k))=(y_j+c_j)-(y_i+c_i)\in\Xi-\Xi$.
Also, since $T,T'$ are regionally proximal with respect to the $V$-action, by definition of the regional proximal relation, there are, for each $k\in\N$, tilings $S_k,S'_k\in X$, $v_k\in V$ and $w_k\in B_{\frac{1}{k}}(0)$ with 
$B_0[T]=B_0[S_k]$, $B_0[T']=B_0[S'_k]$, and $B_0[S_k-v_k]=B_0[S'_k-v_k]-w_k$.
From this we deduce that $(y_i+c_i)-(y_j+c_j)+w_k\in\Xi-\Xi$. Thus the vectors $w_k$ 
all lie in the uniformly discrete set $\Xi-\Xi+\Xi-\Xi$. As $w_k\to 0$ it must be that $w_k=0$ for large $k$: let $S=S_k$, $S'=S'_k$, and $v=v_k$ for such $k$.
\end{proof}

\begin{proof}(Of Theorem \ref{PDS}) Let $v_k\in\R^n$ be as hypothesized and let
$E:=\cap_k\text{ span}_{\R}\{v_k,v_{k+1},\ldots\}$. There is then $k$ so that $E=\text{ span}_{\R}\{v_k, v_{k+1},\ldots\}$ and without loss of generality we may assume $k=1$. Now choose $n_i\to\infty$ so that $\Lambda^{n_i}E\to V$. By this we mean that there are bases $\{e_1^i,\ldots,e_b^i\}$ for $\Lambda^{n_i}E$ so that $e_j^i\to e_j$ and 
$\{e_1,\ldots,e_b\}$ is a basis for $V$. We will show that there is a $V$-minimal set $X\subset\OP$ on which the $V$-action is almost automorphic. 

For each $\delta>0$ there is a subset $W_{\delta}={\{w_1(\delta),\ldots,w_b(\delta)}\}$ of $\{v_1,v_2,\ldots\}$ that spans $E$ with $|w_i(\delta)|<\delta$ for each $i$. By translating $\tau$ we may assume that there is $r>0$ with
$B_{3r}(0)$ contained in the support of $\tau$. Given $w\in\R^n$, let $U_w:=\{x\in\R^n:x,x+w\in int(spt(\tau))\text{ and } \tau-w-x\sim_s\tau-x\}$. If $\tau-w-x\sim_s\tau-x$, then $\tau-w-y\sim_s\tau-y$ for all $y$ sufficiently near $x$, so the sets $U_w$ are open.

For a fixed $\delta$, $0<\delta<r$, we see that $U_w$ is dense in $B_{2r}(0)$ for all $w\in W_{\delta}$. Also, $U_{-w}=U_w+w$ is dense in $B_{2r}(0)$ for such $w$. Suppose that $w'$ is such that $|w'|<r$ and $U_{w'}$ is dense in $B_{2r}(0)$ and let $w\in W_{\delta}$. then $\tau-w'-x\sim_s\tau-x$ for a dense open set of $x$ in $B_{2r}(0)$ and $\tau-w-x\sim_s\tau-x$ for a dense open set of $x$ in $B_{2r}(0)$. Thus $\tau-(w'-w)-x\sim_s\tau-x$ for a dense set of $x$ in $B_{2r}(0)$ so that $U_{w'-w}$ is dense in $B_{2r}(0)$. 
Also, $U_{-w'}=U_{w'}+w'$ is dense in $B_{2r}(0)$, so $U_{-w'-w}$ is dense in $B_{2r}(0)$ (as above) and hence $U_{w'+w}=U_{-w'-w}+(w'+w)$ is dense in $B_{2r}(0)$,
provided $|w'+w|<r$. In this way, we see that if $w_{i_1},\ldots,w_{i_k}$ is any sequence in $W_{\delta}$ with the property that $\pm w_{i_1}\pm w_{i_2}\pm\cdots\pm w_{i_l}\in B_r(0)$ for $l=1,\ldots,k$, then $U_w$ is dense in $B_{2r}(0)$, where $w=\pm w_{i_1}\pm\cdots\pm w_{i_k}$. 
This set of $w$ contains all points of the lattice $\text{span}_{\Z}W_{\delta}$ that lie in $B_r(0)$. Thus we have established that :
\begin{equation}\label{dense} \{w\in E:U_w\text{ is dense in }B_{r}(0)\}\text{ is dense in } E\cap B_{r}(0).
\end{equation}

By Lemma \ref{full meas} and the fact that $\{T:\sharp g{-1}(g(T))=cr(\Phi)\}$ has full measure, there is a $V$-minimal set $X$ and a $T\in X$ so that $\sharp g^{-1}(g(T))=cr(\Phi)$. We now argue that if $g_X:X\to X_{max}$ is the maximal equicontinuous factor map for the $V$-action on $X$, then $g_X^{-1}(g_X(T))=\{T\}$.
Suppose that $g_X(T')=g_X(T)$. Then $T'$ is regionally proximal with $T$ under the $V$-action so that, by Lemma \ref{rpV}, there are $S,S'\in X$ and $v\in V$ with $B_0[T]=B_0[S]$, $B_0[T']=B_0[S']$, and $B_0[S-v]=B_0[S'-v]$. Fix $R>|v|$. By repetitivity of patches, there is $R'$ large enough so that if $P$ is any tiling in $\OP$, then $B_{R'}[P]$ contains translates of $B_R[S]$ and $B_R[S']$. For sufficiently large $i$, $\Lambda^{n_i}B_r(0)\supset B_{2R'}(0)$. Since $\tau-x\sim_s\tau\implies\Phi^{n_i}(\tau)-\Lambda^{n_i}x\sim_s\Phi^{n_i}(\tau)$, we have from \ref{dense} above that:
\begin{equation}\label{dense2} \{w\in \Lambda^{n_i}E:U^i_w\text{ is dense in }B_{2R'}(0)\}\text{ is dense in }\Lambda^{n_i}E\cap B_{2R'}(0)
\end{equation}
where $U^i_w:=\{x\in\R^n:x,x+w\in int(spt(\Lambda^{n_i}\tau))\text{ and } \Phi^{n_i}(\tau)-w-x\sim_s\Phi^{n_i}(\tau)-x\}$.

To simplify notation, let $B:=B_R[S]$, $B':=B_R[S']$, and $Q:=B_0[S-v]+v=B_0[S'-v]+v$. For
large $i$ there are then $y=y_i$ and $y'=y'_i$ so that the patches $B+y$ and $B'+y'$
are sub patches of $\Phi^{n_i}(\tau)$ and have supports contained in $B_{R'}(0)$.
Note that $y+v\in int(spt(Q+y))$.
Since $\Lambda^{n_i}E\to V$, we may take $i$ large enough so there is $w\in\Lambda^{n_i}E\cap B_{R'}(0)$ close enough to $v$ so that $y+w\in int(spt(Q+y))$; that is, $w\in int(spt(Q))$. Moreover, by \ref{dense2}, we can choose this $w$ so that $U^i_w$ is dense in $B_{2R'}(0)$. Since $U^i_w$ is also open and $|y|,|y'|\le R'$, $(U^i_w-y)\cap (U^i_w-y')$ is dense in $B_{R'}(0)$ and there is $x$ so that:

\begin{enumerate}
\item $x\in U^i_w-y$,
\item $x\in U^i_w-y'$,
\item $x\in int(spt(B_0[T]))$,
\item $x\in int(spt(B_0[T']))$, and
\item $x\in int(spt(Q-w))$.
\end{enumerate}
Let $P:=\Phi^{n_i}(\tau)$. Since $Q+y$ and $Q+y'$ are both patches in $P$ and $0\in int(spt(Q-w))$, we have from (5) that: $$P-w-y-x\sim_s P-w-y'-x.$$
From (1) we get: $$P-y-x\sim_s P-w-y-x$$
and from (2): $$P-y'-x\sim_s P-w-y'-x.$$
Hence $$P-y-x\sim_s P-y'-x$$
and from (3) and (4) we have $B_0[T]-x\sim_s B_0[T']-x$; that is, $T-x\sim_s T'-x$.
But $T'\in g^{-1}(g(T))$ and $\sharp g^{-1}(g(T))=cr(\Phi)$. From (2) of Theorem \ref{tools}
it must be the case that $T'=T$. In other words, $\sharp g_X^{-1}(g_X(T))=1$ and the $V$-action on $X$ has minimal rank 1.

\end{proof}

We return now to consideration of the relation $R$.

\begin{lemma}\label{bound on R} Suppose that $\Phi$ is an $n$-dimensional substitution of Pisot family type that forces its border. If $\OP$ has no almost automorphic sub actions, there is $B<\infty$ so that $x_iRx_{i+1}$ for $i=1,\ldots,n-1\implies \sharp\{x_1,\ldots,x_n\}\le B$.
\end{lemma}
\begin{proof} Let us say that $x\bar{R}y$ if there are $n\in\N$ and $z_1,\ldots, z_n$ so that $z_1=x$, $z_n=y$ and $z_jRz_{j+1}$ for $j=1,\ldots, n-1$. For such $z_j$, let $T_j,T_j'\in\OP$ be such that $T_j\approx_s T_j'$, $\pi(T_j)=z_j$ and $\pi(T_j')=z_{j+1}$. Since $\Phi$ forces its border, there is $\epsilon>0$ so that
if $\pi(T)=\pi(T')$, then $T-v\sim_sT'-v$ for all $v$ with $|v|<\epsilon$. Thus, $T_j-v\sim_sT_{j-1}'-v$ for each $j>1$ and $|v|<\epsilon$, and $T_j-v\sim_sT_j'-v$ for each $j\ge1$ and an open dense set of $v\in B_{\epsilon}(0)$. We see that $x\bar{R}y$, $\pi(T)=x$ and $\pi(T')=y$ implies that $T-v\sim_sT'-v$ for an open dense set of $v$ in $B_{\epsilon}(0)$.

Now, if there is no such $B$ as in the statement of the lemma, we may find two sequences $\{x_i\}_{i\in\N},\{y_i\}_{i\in\N}$ with the properties: $x_i$ and $y_i$ are in the interior of the same $n$-cell of $X_{\Phi}$ for all $i$; $x_i\ne y_i$ for all $i$; $x_i\bar{R}y_i$ for all $i$; and $x_i,y_i\to x\in X_{\Phi}$ as $i\to\infty$. Since the $x_i,y_i$ are in the interior of the same $n$-cell, there is $T\in\OP$ with $\pi(T)=x$ and $r_i,s_i\in spt(B_0[T])$ with $\pi(T-r_i)=x_i$ and $\pi(T-s_i)=y_i$. Without loss of generality, assume that $|s_i|<\epsilon/2$. We have that $T-r_i-v\sim_sT-s_i-v$ for a dense set (and open) set of $v\in B_{\epsilon}(0)$. Then $T-(r_i-s_i)-v\sim_sT-v$ for a set $S_i$ of $v$ that is open and dense in $B_{\epsilon/2}(0)$. Then $T-(r_i-s_i)-v\sim_sT-v$ for a dense set of $v$ in $B_{\epsilon/2}(0)$ (namely,  for $v$ in $B_{\epsilon/2}(0)\cap\cap_iS_i$). Let $N$ be large enough so that the support of $Q:=B_0[\Phi^N(T)]$ is contained in $\Lambda^NB_{\epsilon/2}(0)$. Then $Q$ and $Q-v_i$ are densely stably related on overlap for $v_i:=\Lambda^N(r_i-s_i)$. By Theorem \ref{PDS},  $\OP$ has an almost automorphic sub action.
\end{proof}

Suppose that the $\R^n$-action on $\OP$ does not have an almost automorphic sub action.
Define $\asymp_s$ on $X_{\Phi}$ by $x\asymp_s x'$ iff there are $x_1,\ldots,x_n\in X_{\phi}$ with $x=x_1,\, x'=x_n$, and $x_iRx_{i+1}, i=1,\ldots,n-1$. (That is, $\asymp_s$ is the transitive closure of $R$.)

\begin{lemma}\label{equiv rel}
If the Pisot family tiling space $\OP$ does not have an almost automorphic sub action,  then $\asymp_s$ is a closed equivalence relation on $X_{\Phi}$.
\end{lemma}
\begin{proof} By definition, $\asymp_s$ is an equivalence relation. To see that it's closed, we first show that the relation $R$ is closed. Suppose that $x^i,y^i\in X_{\Phi}$ are such that $x^i\to x$, $y^i\to y$, and $x^iRy^i$ for $i\in\N$. There are then $T_i,T'_i\in\OP$ with $T_i\approx_sT'_i$ and $\pi(T_i)=x^i,\pi(T'_i)=y^i$. Passing to a subsequence,
we may suppose that $T_i\to T, T_i'\to T'$. Then $\pi(T)=x$, $\pi(T')=y$, and since, as observed earlier, $\approx_s$ is closed, $T\approx_s T'$. That is, $xRy$, and $R$ is closed.

Now suppose that $x^i,y^i, x, y$ are as above , except $x^i\asymp_sy^i$. By Lemma \ref{bound on R} there are $B\in\N$ and $z^i_j$, $j\in\{1,\ldots,B\}, i\in\N$, with $x^i=z^i_1$, $y^i=z^i_B$ and $z^i_jRz^i_{j+1}$ for $1\le j\le B-1$. Passing to a subsequence, we may suppose that $z^i_j\to z_j$ for $1\le j\le B$ as $i\to\infty$. Then,
from the above, $z_jRz_{j+1}$ for $1\le j\le B-1$, so $x\asymp_s y$ and $\asymp_s$ is closed.
\end{proof}

If $\tau$ is a tile for $\Phi$ and $v\in spt(\tau)$, by $\pi(\tau-v)$ we mean $\pi(T-v)$ for any $T\in\OP$ with $\tau\in T$.
We call a point $x\in X_{\Phi}$ an {\em inner point} if each element of $[x]$ is an interior point of its $n$-face, and if $v\in spt(\tau)$, $\tau$ a tile for $\Phi$ , and $\pi(\tau-v)$ is an inner point, we'll say that {\em $v$ is an inner point of $\tau$}. A {\em stack} is a collection $\bar{\tau}=\{\tau_1,\ldots,\tau_k\}$ of tiles with the property that $spt(\bar{\tau}):=\cap_{i=1}^k\tau_i$ has nonempty interior and for each $v\in int(spt(\bar{\tau}))$,
$[\pi(\tau_j-v)]=\{\pi(\tau_1-v),\ldots,\pi(\tau_k-v)\}$ for some (equivalently, all) $j\in\{1,\ldots,k\}$. Note that each inner point $v$ of a tile $\tau$ uniquely determines a stack
$s(\tau,v):=\{\text{ tiles }\sigma: v\in \text{ spt}(\sigma) \text{ and } \pi(\sigma-v)\asymp_s\pi(\tau-v)\}$. It follows from  (4) of Theorem \ref{tools} and Lemma \ref{bound on R} that
there are only finitely many translation equivalence classes of stacks. Each tile $\tau$ for $\Phi$ is thus tiled by the finitely many stacks $s(\tau,v),\, v$ an inner point of $\tau$.

We now define a substitution $\Phi_s$ on stacks. Let $\bar{\tau}=\{\tau_1,\ldots,\tau_k\}$ be a stack with support $e=\cap_{i=1}^kspt(\tau_i)$ and let $\Phi(\tau_1)=\{\sigma_1,\ldots,\sigma_l\}$. Then $$\Phi_s(\bar{\tau}):=\{s(\sigma_j,v): v\in\Lambda e, v \text{ an inner point of }\sigma_j, j=1,\ldots,l\}.$$ Note that since $f_{\Phi}$ preserves $\asymp_s$ the definition of $\Phi_s$ is not altered by replacing $\tau_1$ by $\tau_i$ for any $i\in\{1,\ldots k\}$. Let $\Sigma:\OP\to \Omega_{\Phi_s}$ by $$\Sigma(T):=\{s(\tau,v):\tau\in T,\, v\text{ an inner point of }\tau\}.$$

\begin{Theorem}\label{minimal model} Suppose that $\Phi$ is  an $n$-dimensional Pisot family substitution that forces the border and has convex tiles. Suppose also that $\OP$ does not have an almost automorphic sub action. Let $g:\OP\to \T_{\max}$ be the map onto the maximal equicontinuous factor. Then: 
\begin{enumerate}
\item $\Phi_s$ is Pisot family, forces the border, and is primitive and non-periodic;
\item The map $\Sigma:\OP\to \Omega_{\Phi_s}$ is $a.e.$ one-to-one, and semi-conjugates $\Phi$ with $\Phi_s$ and the $\R^n$-action on $\OP$ with that on $\Omega_{\Phi_s}$;
\item $\Omega_{\Phi_s}$ is homeomorphic with $\OP/\approx_s$ by a homeomorphism that conjugates both translation and substitution dynamics and the 
relation $\approx_s$ on $\Omega_{\Phi_s}$ is trivial;
\item The map $g$ factors as $g=g_s\circ \Sigma$, where $g_s:\Omega_{\Phi_s}\to \T_{max}$ is the maximal equicontinuous factor map of $\Omega_{\Phi_s}$.
\end{enumerate}
\end{Theorem}
\begin{proof}

Suppose that $T,S\in\OP$ are such that $\pi(T)\asymp_s\pi(S)$. There are then
$T_i,S_i\in\OP$ with
$\pi(T)=\pi(T_1)$, $T_i\approx_s S_i$, $\pi(T_{i+1})=\pi(S_i)$, $i=1,\ldots, k-1$, and $\pi(S_k)=\pi(S)$. Since $\Phi$ forces the border, there is $\epsilon>0$ and $m\in N$ so that if $T',S'\in\OP$ and $\pi(T')=\pi(S')$, then $\pi(\Phi^m(T'-v))=\pi(\Phi^m(S'-v))$ for all $v\in B_{\epsilon}(0)$. Together with the fact that $\Phi$ preserves $\approx_s$,
this means that $$\pi(T)\asymp_s\pi(S)\implies \pi(\Phi^m(T-v))\asymp_s\pi(\Phi^m(S-v))$$ for all $v\in B_{\epsilon}(0)$.

Towards proving that $\Phi_s$ forces the border, let's examine what it means for stacks to be adjacent in an element of $\Omega_{\Phi_s}$.  First suppose that stacks $s(\sigma_1,v_1)$ and $s(\sigma_2,v_2)$ in $\Phi_s(\bar{\tau})$ (see the definition of $\Phi_s$) are distinct and meet along an $(n-1)$-cell. Say $v\in spt(s(\sigma_1,v_1))\cap spt(s(\sigma_1,v_1))$. Then there is a tiling $T\in\OP$ and tiles $\eta_i\in T$, $\eta_i\in s(\sigma_i,v_i)$, $i=1,2$, with $v\in spt(\eta_1)\cap spt(\eta_2)$. Note that $s(\sigma_i,v_i)=s(\eta_i,v_i)$, $i=1,2$. This structure persists under application of $\Phi_s$ and we see, since every allowed adjacency occurs in some $\Phi^k(\bar{\tau})$,
that if $\bar{\sigma}$ and $\bar{\sigma}'$ are adjacent stacks in an element of $\Omega_{\Phi_s}$ with $v$ in the intersection of their supports, then there is $T\in \OP$
and tiles $\eta,\eta'\in T$ so that $\bar{\sigma}=s(\eta,w)$, $\bar{\sigma}'=s(\eta',w')$, with $w,w'$ as close to $v$ as we wish.

Suppose now that $\bar{S_1},\bar{S_2}\in\Omega_{\Phi_s}$ are such that $\bar{\eta}\in\bar{S_1}\cap\bar{S_2}$, and there are stacks $\bar{\sigma}_1\in\bar{S_1}$, $\bar{\sigma}_2\in\bar{S_2}$ that meet in an $(n-1)$-cell contained in the boundary of the support of  
$\bar{\eta}$; say $v$ is in this common boundary. There are then $T,S\in\OP$ and $\eta_1,\sigma_1\in T$, $\eta_2,\sigma_2\in S$, with $s(\eta_1,w)=s(\eta_2,w)=\bar{\eta}$, 
$s(\sigma_1,w')=\bar{\sigma}_1$, $s(\sigma_2,w')=\bar{\sigma}_2$, and $w,w'$ as close to $v$ as desired. Since, $w$ is as close to v as we wish, $\pi(T-w)=\pi(\eta_1-w)\asymp_s\pi(\eta_2-w)=\pi(S-w)$, and $\asymp_s$ is closed, we have $\pi(T-v)\asymp_s\pi(S-v)$. Take $w'$ with $|w'-v|<\epsilon$. Then $\pi(\Phi^m(\sigma_1-w'))=\pi(\Phi^m(T-w'))\asymp_s\pi(\Phi^m(S-w'))=\pi(\Phi^m(\sigma_2-w'))$. Taking $w'$ close enough to $v$ so that tiles $\sigma_i'\in\Phi^m(\sigma_i)$ with $\Lambda^mv$ in the boundary of their supports, $i=1,2$, also contain $\Lambda^mw'$ in their supports, we have $s(\sigma_1',\Lambda^mw')=s(\sigma_2',\Lambda^mw')$. That is, The stacks in $\Phi_s^m(\bar{S}_1)$ and $\Phi_s^m(\bar{S}_2)$ that meet $\Phi_s^m(\bar{\eta})$ at $\Lambda^mv$ are identical. In other words, $\Phi_s$ forces the border.
We will prove non-periodicity later.

Let $\hat{\Phi}_s$ denote the shift homeomorphism on $\inv \Phi_s$ and let $\hat{\Sigma}:\OP\to\inv\Phi_s$ denote the map induced by $\Sigma$. We will prove that
$\hat{\Sigma}(T)=\hat{\Sigma}(T')\Leftrightarrow T\approx_s T'$. If $T\approx_s T'$, then 
$\Phi^{-k}(T)\approx_s\Phi^{-k}(T')$ for all $k\in\N$, so $\pi(\Phi^{-k}(T)-v)\asymp_s\pi(\Phi^{-k}(T')-v)$ for all $v\in\R^n$, $k\in\N$. Then $\Sigma(\Phi^{-k}(T))=\Sigma(\Phi^{-k}(T'))$ for all $k\in\N$ and $\hat{\Sigma}(T)=\hat{\Sigma}(T')$.

For the converse, first note that if $\Sigma(S)=\Sigma(S')$, then $S-v\sim_s S'-v$ for a dense set of $v\in\R^n$. Indeed, If $\bar{\sigma}$ is a any stack with $\eta,\tau\in \bar{\sigma}$ so that $\pi(\eta-v) R\pi(\tau-v)$ for $v\in int(spt(\eta)\cap spt(\tau))$, then 
$\eta-v\sim_s\tau-v$ for a dense open set of such $v$. Thus if $\sigma\in (\bar{\sigma}\cap S)$ and $\sigma'\in (\bar{\sigma}\cap S')$, $\sigma-v\sim_s\sigma'-v$ for a dense set of $v$ in $spt(\bar{\sigma})$, and hence $S-v\sim_s S'-v$ for a dense set of $v\in\R^n$. 

Suppose that $\hat{\Sigma}(T)=\hat{\Sigma}(T')$. We have that $\Phi^{-k}(T)$ and $\Phi^{-k}(T')$ are densely stably related for all $k\in\N$. Since there are only finitely many stacks up to translation, there are tiles $\tau,\tau'$ with $int(spt(\tau\cap\tau'))\ne\emptyset$ and $v_k\in\R^n$ so that $\tau-v_k\in\Phi^{-k}(T)$ and $\tau'-v_k\in\Phi^{-k}(T')$ for all $k\in\N$. Pick $w\in int(spt(\tau)\cap(spt(\tau')))$ with $\tau-w\sim_s\tau'-w$; say $\eta\subset \Phi^l(\tau)\cap\Phi^l(\tau')$. Then $\Phi^{(k-l)}(\eta)-\Lambda^kv_k\in T\cap T'$ for all $k\ge l$. This means that $T$ and $T'$ are proximal under the $\R^n$ action. Hence $g(T)=g(T')$ so that $T\approx_s T'$. 

Thus we may identify $\inv \Phi_s$ with $\Omega/\approx_s$. 

We argue now that $\Phi_s$ is non-periodic. Suppose not. If $0\ne p\in \R^n$ and $\bar{T}\in\Omega_{\Phi_s}$ are such that $\bar{T}-p=\bar{T}$, then $\bar{S}-p=\bar{S}$ for all $\bar{S}\in\Omega_{\Phi_s}$ (by minimality of the $\R^n$-action). Let $V:=span_{\R}\{p:p \text{ is a period for the translation action}\}$. Then $V$ is a nontrivial subspace of $\R^n$, $\Lambda V=V$, and each $V$-orbit, $\{\bar{T}-v:v\in V\}$, is a $k$-torus ($k=$ dim$(V)$). It follows that every $V$-orbit closure in $\OP/\approx_s\simeq\inv \Phi_s$ is minimal and equicontinuous (it's the $\R^k$-action on a $k$-solenoid or torus). Since $\OP$ does not have an almost automorphic sub action, $cr(\Phi)>1$.
As in the proof of Theorem \ref{PDS}, there must be a $V$-minimal set $X$ in $\OP$ containing a point $T$ with $\sharp g^{-1}(g(T))=cr(\Phi)$. Let $A:\OP\to\Omega_{\Phi}/\approx_s$ denote the quotient map. Since
$A|_X$ factors the $V$-action onto an equicontinuous action, if  $\sharp( (A|_X)^{-1}(A(T)))=1$ then the $V$-action on $X$ has minimal rank 1, contrary to the hypothesis. Thus there must be $T'\ne T\in X$ with $A(T')=A(T)$. But then $g(T')=g(T)$ and $T'-w\sim_s T-w$ for dense $w\in\R^n$. This is not possible (by (2) of Theorem \ref{tools} and $\sharp g^{-1}(g(T))=cr(\Phi)$). Thus, $\Phi_s$ is non-periodic.

It follows from non-periodicity that $\Phi_s:\Omega_{\Phi_s}\to\Omega_{\Phi_s}$ is a homeomorphism (\cite{Sol}). This means that $\Omega_{\Phi_s},\,\inv\Phi_s$, and
$\OP/\approx_s$ are isomorphic. Since $\Sigma(T)=\Sigma(T')\implies T\approx_s T'$, 
$g$ factors through $\Omega_{\Phi_s}$, say $g=g_s\circ\Sigma$, and $g_s$ must (by maximality of $\T_{max}$ for $\OP$) be the maximal equicontinuous factor map of $\Omega_{\Phi_s}$. Finally, it is clear from definitions that if $\Sigma(T)$ and $\Sigma(T')$ are densely stably related under $\Phi_s$ then $T$ and $T'$ are densely stably related under $\Phi$. Also, $g_s(\Sigma(T))=g_s(\Sigma(T'))\implies g(T)=g(T')$. Hence
$\Sigma(T)\approx_s\Sigma(T')$ in $\Omega_{\Phi_s}$ implies that $T\approx_sT'$
in $\OP$, which means $\Sigma(T)=\Sigma(T')$. That is, $\approx_s$ is trivial on $\Omega_{\Phi_s}$.
\end{proof}

\begin{lemma}\label{distal twin} Suppose that $\Psi$ is  an $n$-dimensional Pisot family substitution with $cr(\Psi)=2$ and suppose that $\approx_s$ is trivial on $\Omega_{\Psi}$. Then for each $T\in\Omega_{\Psi}$ there is a unique
$T'\in\Omega_{\Psi}$ with the properties: $T$ and $T'$ are regionally proximal and
$T\cap T'=\emptyset$.
\end{lemma}
\begin{proof} From (4) of Theorem \ref{tools} there are, up to translation, only finitely many 
pairs of patches $(B_0[T],B_0[T'])$ with $T$ and $T'$ regionally proximal. It follows that there is $k\in\N$ so that if $T$ and $T'$ are regionally proximal and $\Psi^l(T)\cap\Psi^l(T')\ne\emptyset$ for some $l\in\Z$, then $\Psi^k(T)\cap\Psi^k(T')\ne\emptyset$. From this we see that if $T$ and $T'$ are regionally proximal and $T\cap T'=\emptyset$, then $\Psi^l(T)\cap\Psi^l(T')=\emptyset$ for all $l\in\Z$.

From (2) of Theorem \ref{tools}, and $cr(\Psi)=2$, we have that for each $T$ there is $T'$ with $T$ and $T'$ regionally proximal and $T\cap T'=\emptyset$. Suppose that $T''$ is also regionally proximal with $T$ and $T''\cap T=\emptyset$. From the above, $\Psi^l(T')\cap\Psi^l(T)=\emptyset$ and $\Psi^l(T'')\cap\Psi^l(T)=\emptyset$ for all $l\in\Z$. Suppose that there are $v_0$ and $\epsilon>0$ with $T'-v\nsim_sT''-v$ for all $v\in B_{\epsilon}(v_0)$. Pick $l_k\to\infty$ with $\Psi^{l_k}(T-v_0)\to S$,
$\Psi^{l_k}(T'-v_0)\to S'$, and $\Psi^{l_k}(T''-v_0)\to S''$. Then $S\cap S'=\emptyset$,
$S\cap S''=\emptyset$, $S'\cap S''=\emptyset$, and $S,S',S''$ are all regionally proximal. 
But then, by (2) of Theorem \ref{tools}, $cr(\Psi)\ge3$. So there can be no such $v_0,\epsilon$ and we conclude $T''\approx_s T'$. But $\approx_s$ is trivial on $\Omega_{\Psi}$, so $T''=T'$.      
\end{proof}

Let us call a pair $(T,T')$ as in Lemma \ref{distal twin} an $rpd$ ordered pair (short for {\em regionally proximal distal}). 
Let $(T,T')$ be any $rpd$ ordered pair. Up to translation, there are only finitely many ordered pairs $(\tau,\tau')$ with $\tau\in T,\tau'\in T'$, and $int(spt(\tau)\cap spt(\tau'))\ne\emptyset$. We'll call such a pair an $rpd$ {\em ordered tile pair} and its support is
$spt((\tau,\tau')):=spt(\tau)\cap spt(\tau')$. (We will assume, again for convenience, that the supports of tiles for $\Psi$ are convex balls so the same is true for the $rpd$ tile pairs.) $\Psi$ naturally induces a substitution $\Psi_{op}$ on the $rpd$ ordered tile pairs and it is easy to see that if $\Psi$ forces the border, then so does $\Psi_{op}$. Let
$\Sigma_{op}:\Omega_{\Psi}\to\Omega_{\Psi_{op}}$ be given by $$\Sigma_{op}(T)= \{(\tau,\tau'):(\tau,\tau') \text{ is an } rpd \text{ ordered tile pair }, \,\tau\in T,\, \tau'\in T',\, (T,T') \text{ an $rpd$ ordered pair}\}.$$

\begin{lemma}\label{op} Suppose that $\Psi$ is  an $n$-dimensional Pisot family substitution with $cr(\Psi)=2$ and suppose that $\approx_s$ is trivial on $\Omega_{\Psi}$. Then $\Sigma_{op}:\Omega_{\Psi}\to\Omega_{\Psi_{op}}$ is an isomorphism.
\end{lemma}
\begin{proof} $\Sigma_{op}$ is the composition of isomorphisms $T\mapsto (T,T')$,
from $\Omega_{\Psi}$ to $\{(T,T')\in\Omega_{\Psi}\times\Omega_{\Psi}: (T,T') \text{ is an $rpd$ ordered pair}\}$, with $(T,T')\mapsto\{(\tau,\tau'):\tau\in T, \tau'\in T', int(spt(\tau)\cap spt(\tau'))\ne\emptyset\}$ from $\{(T,T')\in\Omega_{\Psi}\times\Omega_{\Psi}: (T,T') \text{ is an $rpd$ ordered pair}\}$ to $\Omega_{\Psi_{op}}$.
\end{proof}

We may carry out an analogous construction, replacing $rpd$ ordered pairs $(T,T')$, by 
(unordered) $rpd$ pairs $\{T,T'\}$. The tiles, which we call {\em rpd tile pairs}, are now
pairs $\{\tau,\tau'\}$ such that $(\tau,\tau')$ is an ordered $rpd$ tile pair, with $spt(\{\tau,\tau'\}):=spt(\tau)\cap spt(\tau')$. Let $\Psi_p$ denote the natural substitution induced on $rpd$ tile pairs by $\Psi$.  Let $\sim_{rpd}$ denote the equivalence relation:
$T\sim_{rpd}T'\Longleftrightarrow T=T'$ or $(T,T')$ is an $rpd$ ordered pair. Also, let $F$ denote the morphism from $rpd$ ordered tile pairs to $rpd$ tile pairs that forgets order:
$F((\tau,\tau')):=\{\tau,\tau'\}$. Then $F$ induces a map, also denoted by $F$, from $\Omega_{\Psi_{op}}$ onto $\Omega_{\Psi_p}$. Let $\Sigma_p:=F\circ\Sigma_{op}:\Omega_{\Psi}\to\Omega_{\Psi_p}$. Finally, let $\gamma$ denote the involution $\gamma((\tau,\tau')):=(\tau',\tau)$, and let $\Gamma$ denote the involution of $\Omega_{\Psi_{op}}$  (and, abusing notation, of $\Omega_{\Psi}$) induced by $\gamma$. Note that $\Gamma\circ\Psi_{op}=\Psi_{op}\circ\Gamma$ (and $\Gamma\circ\Psi=\Psi\circ\Gamma$).

\begin{lemma}\label{rpd p space} Suppose that $\Psi$ is  an $n$-dimensional Pisot family substitution that forces the border with $cr(\Psi)=2$ and suppose that $\approx_s$ is trivial on $\Omega_{\Psi}$. Then $\Psi_p$ is non-periodic, forces the border, and $\Omega_{\Psi_p}$ is isomorphic with $\Omega_{\Psi}/\sim_{rpd}$.
\end{lemma}
\begin{proof} It is clear that $\Psi_p$ forces the border. Suppose that there is $0\ne v\in\R^n$ and $T\in\Omega_{\Psi}$ with $\Sigma_{op}(\bar{T})-v=\Sigma_{op}(\bar{T})$. 
Since $\Omega_{\Psi_{op}}$ has no non-zero periods (being isomorphic with $\Omega_{\Psi}$), there are adjacent $(\tau,\tau'),(\sigma,\sigma')\in\bar{T}$, say $x\in
spt((\tau,\tau'))\cap spt((\sigma,\sigma'))$, such that $(\tau,\tau')-v,(\sigma',\sigma)-v\in\bar{T}$. Since $\Psi_{op}$ forces the border, there are $m\in\N$, $(\alpha,\alpha')\in\Phi_{\Psi_{op}}^m(\bar{T})$, so that the $rpd$ ordered tile pairs meeting $\Psi_{op}^m((\tau,\tau'))$ and $\Psi_{op}^m((\tau,\tau')-v)$, at $\Lambda^m x$, resp., $\Lambda^m(x-v)$, and contained in $\Psi_{op}^m((\sigma,\sigma'))$, resp., $\Psi_{op}^m((\sigma,\sigma')-v)$, are $(\alpha,\alpha')$ and $(\alpha,\alpha')-\Lambda^m v$, resp.. But $\Psi_{op}^m((\sigma',\sigma))=\Gamma(\Psi_{op}^m(\sigma,\sigma'))$. This means that $(\alpha',\alpha)=(\alpha,\alpha')$ so that $\alpha'=\alpha$. This violates the `distal' part of being an $rpd$ ordered pair. Thus $\Psi_p$ is non-periodic and, by \cite{Sol}, $\Psi_p:\Omega_{\Psi_p}\to\Omega_{\Psi_p}$ is a homeomorphism.

If $T\sim_{rpd}T'$, then certainly $\Sigma_p(T)=\Sigma_p(T')$. Conversely, if $\Sigma_p(T)=\Sigma_p(T')$, then $\Sigma_p(\Psi^{-k}(T))=\Psi_p^{-k}(\Sigma_p(T))=\Psi_p^{-k}(\Sigma_p(T'))=\Sigma_p(\Psi^{-k}(T'))$ for all $k\in\N$. Suppose that $T\ne T'$ and choose $\{\tau,\tau'\}\in \Sigma_p(T)=\Sigma_p(T')$ with $\tau\ne\tau'$ and, say, $x\in int(spt(\{\tau,\tau'\}))$. Then, if $\{\sigma,\sigma'\}\in\Psi_p^{-k}(\Sigma_p(T))=\Psi_p^{-k}(\Sigma_p(T'))$ is such that $\Lambda^{-k} x\in int(spt(\{\sigma,\sigma'\}))$, 
$\sigma\ne\sigma'$. Hence all pairs $\{\eta,\eta'\}$ in $\Psi_p^k(\{\sigma, \sigma'\})$ satisfy $\eta\ne\eta'$. Since $k\in\N$ is arbitrary, and $\Psi_p$ forces the border, $\eta\ne\eta'$ for all $\{\eta,\eta'\}\in \Sigma_p(T)=\Sigma_p(T')$. That is, $(T,T')$ is an $rpd$ pair.
Thus $\Sigma_p(T)=\Sigma_p(T')\Leftrightarrow T\sim_{rpd}T'$, and $\Omega_{\Psi_p}\simeq\Omega_{\Psi}/\sim_{rpd}$.
\end{proof}

As the following example shows, it is important in the construction of $\Psi_p$ that the substitution $\Psi$ forces the border.

\begin{example}\label{TM} Consider the Thue-Morse substitution $\psi:a\mapsto ab,b\mapsto ba$.
(The corresponding tile sustitution $\Psi$ does not force the border.) The relation $\approx_s$ is trivial on $\Omega_{\Psi}$,
and $cr(\Psi)=2$. If $T\in\Omega_{\Psi}$, then the $rpd$ twin $T'$ of $T$ is obtained by switching the labels of all tiles in $T$. If we form $\Psi_p$ as above we obtain only one prototile and $\Psi_p$ degenerates to the tile substitution corresponding to a periodic substitution $a\mapsto aa$. 

If instead we first collar $\Psi$, so that the resulting substitution $\Phi$ forces the border, the $\Phi_p$ obtained is (a version of) the tile substitution corresponding to {\em period-doubling}, $a\mapsto ab,b\mapsto aa$,  and the space $\Omega_{\Phi_p}$ is isomorphic with 
$\Omega_{\Psi}/\sim_{rpd}$ (a 2-adic solenoid with exactly one arc-component split into a bi-asymptotic pair).

\end{example}

\begin{Theorem}\label{rpd p space} Suppose that $\Psi$ is  an $n$-dimensional Pisot family substitution that forces the border with $cr(\Psi)=2$ and suppose that $\approx_s$ is trivial on $\Omega_{\Psi}$. Then $\Sigma_p:\Omega_{\Psi}\to\Omega_{\Psi_p}$ is a 2-to-1 covering map that semi-conjugates the dynamics,
$\Omega_{\Psi_p}$ is a Pisot family substitution tiling space whose $\R^n$-action has pure discrete spectrum, and the maximal equicontinuous factor map on $\Omega_{\Psi}$ factors, via $\Sigma_p$, through $\Omega_{\Psi_p}$. 
\end{Theorem}

\begin{proof}
Identifying $\Omega_{\Psi}$ with $\{(T,T')\in\Omega_{\Psi}\times\Omega_{\Psi}:T\ne T'\text{ and }T\sim_{rpd}T'\}$ (via Lemma \ref{op}) and $\Omega_{\Psi_p}$ with $\{\{T,T'\}\subset\Omega_{\Psi}:T\ne T'\text{ and }T\sim_{rpd}T'\}$, the map $\Sigma_p$ is simply $(T,T')\mapsto\{T,T'\}$
which clearly semi-conjugates the dynamics and is everywhere exactly 2-to-1. Given an $rpd$ ordered pair $(T,T'), T\ne T'$ there is $\delta>0$ so that $d(T-v,T'-v)\ge\delta$ for all $v\in\R^n$. It follows from minimality of the $\R^n$-action on $\Omega_{\Psi}$ that, if $(S,S'), S\ne S'$ is any $rpd$ ordered pair, then $d(S,S')\ge\delta$, and from this, that small neighborhoods in $\Omega_{\Psi_p}$ are evenly covered; that is, $\Sigma_p$ is a 2-to-1 covering map.

If $(T,T'), T\ne T'$ is an $rpd$ ordered pair and $r>0$, there are $S_1(r),S_2(r)\in\Omega_{\Psi}$ and $v_r$ so that $B_r[T]=B_r[S_1(r)]$, $B_r[T']=B_r[S_2(r)]$ and $B_r[S_1(r)-v_r]=B_r[S_2(r)-v_r]$
((6) of Theorem \ref{tools}). Let $(S_1(r),S'_1(r)), S_1(r)\ne S'_1(r)$ and $(S_2(r),S'_2(r)), S_2(r)\ne S'_2(r)$ be $rpd$ ordered pairs. By uniqueness of $rpd$ partners, $S'_1(r)\to T'$ as $r\to\infty$. Thus, given $R>0$, $B_R[T']=B_R[S_1'(r)]$ for all sufficiently large $r$. Similarly, $B_R[T]=B_R[S'_2(r)]$ and $B_R[S'_1(r)-v_r]=B_R[S'_2(r)]$ for large $R$. Let $\bar{T}$, $\bar{T}'$, $\bar{S}$, and $\bar{S}'$ denote the elements of $\Omega_{\Psi_{op}}$ corresponding to $(T,T')$, $(T',T)$, $(S_1(r),S'_1(r))$, and $(S_2(r),S'_2(r))$, resp. Then
$B_R[\bar{T}]=B_R[\bar{S}]$, $B_R[\bar{T'}]=B_R[\bar{S'}]$, and $B_R[\bar{S}-v_r]=B_R[\bar{S}'-v_r]$, and we see that $g(\bar{T})=g(\bar{T}')$, where $g$ is the maximal equicontinuous factor map on $\Omega_{\Psi}$ (again, by (6) of Theorem \ref{tools}). It follows that there is a semi-conjugacy $h$ so that $g=h\circ\Sigma_p$. Since $g$ is a.e 2-to-1 and $\Sigma_p$ is exactly 2-to-1, $h$ is a.e. 1-1. Then $h$ is the maximal equicontinuous factor map on $\Omega_{\Psi_p}$ and the $\R^n$-action on $\Omega_{\Psi_p}$ has pure discrete spectrum, by (5) of Theorem \ref{tools}.
\end{proof}

 In \cite{BGG} the authors introduce the idea of a maximal pure discrete factor; that is, a factor on which the action is pure discrete which is not properly contained in any other such factor. Two examples of this situation given there are the period doubling factor of Thue-Morse (Example \ref{TM} above) and a pure discrete factor of the
2-dimensional squiral tiling space. Both are instances of Theorem \ref{rpd p space}.
In fact, in the setting of Theorem \ref{rpd p space}, there is, up to isomorphism (of $\R^n$-actions), a unique maximal factor strictly between $\Omega_{\Psi}$ and $\T_{max}$ and the construction of $\Omega_{\Psi_p}$ gives an algorithm for finding it.

\begin{prop}\label{max pure discrete} Suppose that $\Psi$ is  an $n$-dimensional Pisot family substitution that forces the border with $cr(\Psi)=2$ and suppose that $\approx_s$ is trivial on $\Omega_{\Psi}$. Then any pure discrete factor of $\Omega_{\Psi}$ that has $\T_{max}$ as a factor is a factor of $\Omega_{\Psi_p}$.
\end{prop}

\begin{proof} 

 Suppose that $\pi_1:\Omega_{\Psi}\to\Omega$ and $\pi_2:\Omega\to\T_{max}$ are factor maps of $\R^n$ actions and that the action on $\Omega$ is pure discrete. By Theorem \ref{rpd p space} we may suppose that $\pi_2$ is proper. There must then be some $T,S$ with $g(T)=g(S)$ and $\pi_1(T)\ne\pi_1(S)$. If $T\sim_{rpd}S$
then (by minimality of the $\R^n$-action and distally of $T$ and $S$) $\pi_1$ is at least 2-1 everywhere and the action on $\Omega$ cannot be pure discrete. Thus $T\sim_{rpd}T'\implies \pi_1(T)=\pi_1(T')$, and $\pi_1$ factors through $\Omega_{\Psi}/\sim_{rpd}\simeq \Omega_{\Psi_p}$.
\end{proof}

\section{One-dimensional Pisot Substitutions}

It will be convenient in what follows to make a (one-dimensional) substitution  force its border by making it {\em proper}, rather than by collaring. To do this we need to pass to a power of the substitution, but this causes no problem as the spaces for a substitution and  a proper version of it have isomorphic $\R$-actions. For details, see \cite{Dur} and \cite{BD}. Briefly,
given a substitution $\phi$, take $n$ so that there are letters $a,b$ with $\phi^n(a)=a\ldots$, $\phi^n(b)=\ldots b$, and so that $ba$ is an allowed word for $\phi$. Let $\{w^1,\dots,w^m\}$ be the collection of all words, each of which starts in $a$ and ends in $b$, such that  $bw^ia$ is allowed for $\phi$. $\phi^n$ then induces a substitution on 
the alphabet $\{w^1,\ldots,w^m\}$ and this is what we mean by a proper version of $\phi$.
It has a unique fixed bi-infinite word, the corresponding tile substitution forces the border,
and we may take a wedge of circles, one for each tile, for the Anderson-Putnam complex. Specifically, for proper $\phi$ with alphabet $\mathcal{A}=\{1\ldots,m\}$ and substitution matrix $M$, we let $\Phi$ denote the corresponding tile substitution on prototiles $\rho_i=[0,\omega_i]$, $i=1,\ldots,m$, $(\omega_1,\ldots,\omega_m)$ a positive left eigenvector for $M$. Then $X_{\phi}:=\cup_{i=1}^m(\rho_i,i)/b$, $b$ the {\em branch point} $b:=\{(t,i):t=0 \text{ or }t=\omega_i,i=1,\ldots,m\}$. Let $f_{\phi}$ denote the map induced by $\Phi$ on $X_{\phi}$. 

For one-dimensional substitutions, $\OP$ has no almost automorphic sub action iff $\OP$ does not have pure discrete spectrum. In this setting, the tile substitution $\Phi_s$ corresponds to a substitution $\phi_s$ on letters. If $\phi$ is proper and $\OP$ does not have pure discrete spectrum, $\phi_s$ is also proper. (One can get this from the proof that $\Phi_s$ forces the border in Theorem \ref{minimal model}, or, directly, by considering the definition of $\Phi_s$ on stacks.)
The natural map from $X_{\phi}$ to $X_{\phi_s}$, taking the quotient by $\asymp_s$, is induced by a morphism $\sigma$ on letters and will be denoted by $f_{\sigma}$. Letting
$[b]:=f_{\sigma}(b)$ denote the branch point of $X_{\phi_s}$, properness of $f_{\phi_s}$ means that all the edge germs of $X_{\phi_s}$ at $[b]$ are flattened by $f_{\phi_s}$ onto a single arc through $[b]$.

\begin{prop} \label{injective on H^1} Suppose that $\phi$ is a proper Pisot substitution such that $\OP$ does not have pure discrete spectrum. Then $\Sigma^*:H^1(\Omega_{\Phi_s})\to H^1(\OP)$ is injective.
\end{prop}
\begin{proof} The proof will use two key observations: (1) flattening of $X_{\phi_s}$ at $[b]$ under $f_{\phi_s}$ permits the lifting of certain 1-chains in the eventual range of $(f_{\phi_s})_*$ to 1-chains in $X_{\phi}$ and (2) while the lifts of cycles might not be cycles, it will follow from the nature of the relation $\asymp_s$ that generates $f_{\sigma}$ - for most $x,x'$, if $f_{\sigma}(x)=f_{\sigma}(x')$, then $f_{\phi}^k(x)=f_{\phi}^k(x')$ for large $k$ -  that, 
under a sufficiently high power of $(f_{\phi})_*$, the lifted cycle becomes a cycle also.
Dualizing the resulting surjection in homology yields the desired injection in cohomology.

To fix notation, let $\pi:\OP\to X_{\phi}$ and $\pi_s:\Omega_{\Phi_s}\to X_{\phi_s}$ be the natural maps with $\pi\circ\Phi=f_{\phi}\circ\pi$ and $\pi_s\circ\Phi_s=f_{\phi_s}\circ\pi_s$, so that $\hat{\pi}:\OP\to \inv f_{\phi}$ and $\hat{\pi}_s:\Omega_{\Phi_s}\to\inv f_{\phi_s}$ are homeomorphisms (by properness) conjugating $\Phi$ with $\hat{f}_{\phi}$ and $\Phi_s$ with $\hat{f}_{\phi_s}$, respectively. Then $\hat{\pi}_s\circ\Sigma=\hat{f}_{\sigma}\circ \hat{\pi}$. For each $i\in\mathcal{A}_{\phi_s}$, we choose a tiling $\bar{T}^i\in\Omega_{\Phi_s}$ so that the $i-th$ petal of $X_{\phi_s}$ is parameterized by $\pi_s(\bar{T}^i-t)$, $0\le t\le\bar{\omega}_i$.

From the definition of the relation $R$ and its transitive closure $\asymp_s$, for each stack $\bar{\tau}=\{\tau_1,\ldots,\tau_l\}$ for $\Phi$, the set of $x$ so that $\tau_i-x\sim_s\tau_j-x$ for all $i,j\in\{1,\ldots,l\}$ is dense in $int(spt(\bar{\tau}))$. For any such $x$ there is $k=k(x)$ so that $B_0[\Phi^{k}(\tau_i-x)]=B_0[\Phi^k(\tau_j-x)]$ for all $i,j\in\{1,\ldots,l\}$. It follows that for each $i\in\mathcal{A}_{\phi_s}$ we may choose $t_i^*\in(0,\bar{\omega}_i)$ so that $x_i^*:=\pi_s(\bar{T}^i-t_i^*)$ has the property that $f_{\phi}^k(f_{\sigma}^{-1}(x_i^*))$ is a singleton for some $k=k(i)\in\N$. Taking $K$ to be the max of the $k(i)$ we have, for each $i\in\mathcal{A}_{\phi_s}$, a point $x_i^*$ in the interior of the $i-th$ petal of $X_{\phi_s}$ so that $f_{\phi}^K(f_{\phi_s}^{-1}(x_i^*))$ is a singleton.

For each $i,j\in\mathcal{A}_{\phi_s}$ let $c_{i,j}$ be the singular 1-chain
in $X_{\phi_s}$ defined by 
$$c_{i,j}(t)= \left\{ \begin{array}{ll}
\pi_{s}(\bar{T}^i-\bar{\omega}_i-t), & \text{ if }
t_i^*-\bar{\omega}_i\le t\le 0 , \\[2mm]
\pi_{s}(\bar{T}^j-t), & \text{ if }0<t\le t_j^* .
\end{array} \right.$$
Every rational 1-chain in $X_{\phi_s}$ is homologous to one of the form $\sum r_{i,j}c_{i,j}$. Since $\phi_s$ is proper, there is $m\in\N$ so that, for each $i,j\in\mathcal{A}_{\phi_s}$, each factor of length two in $\phi_s^m(ij)$ is allowed for $\phi_s$. Then $(f_{\phi_s})^m_*(c_{i,j})$ is homologous to a chain $\sum s_{a,b}^{i,j}c_{a,b}$ with each $ab$ allowed for $\phi_s$. For each $ab$ allowed for $\phi_s$, let $\bar{T}^{ab}\in\Omega_{\Phi_s}$ be such that $\pi_s(\bar{T}^{ab}-t)=\pi_s(\bar{T}^a-\bar{\omega}_a-t)$ for $0\le t\le\bar{\omega}_a$ and  $\pi_s(\bar{T}^{ab}-t)=\pi_s(\bar{T}^b-t)$ for $0\le t\le\bar{\omega}_b$. Let $\tilde{c}_{a,b}(t):=\pi(\bar{T}^{ab}-\bar{\omega}_a-t)$, $t_a^*-\bar{\omega}_a\le t\le t_b^*$. 
Thus, for each $i,j\in\mathcal{A}_{\phi_s}$, there is a rational chain $\sum s_{a,b}^{i,j} \tilde{c}_{a,b}$ in $X_{\phi}$ with $(f_{\sigma})_*(\sum s_{a,b}^{i,j} \tilde{c}_{a,b})=(f_{\phi_s})^m_*(c_{i,j})$. Note that $(f_{\sigma})_*(\tilde{c}_{a,b})=c_{a,b}$ and, while the terminal point of $\tilde{c}_{a,b}$ may not coincide with the initial point of $\tilde{c}_{b,c}$, these points do coincide for $f_{\phi}^K\circ \tilde{c}_{a,b}$ and $f_{\phi}^K\circ\tilde{c}_{b,c}$.

Now if $c=\sum r_{i,j}c_{i,j}$ is a cycle in $X_{\phi_s}$, then $(f_{\phi_s})^{m+K}_*(c)$ is a cycle that is homologous with $(f_{\sigma})_*((f_{\phi})^K_*(\sum_{i,j}r_{i,j}\sum_{a,b} s_{a,b}^{i,j} \tilde{c}_{a,b}))$, and $(f_{\phi})^K_*(\sum_{i,j}r_{i,j}\sum_{a,b} s_{a,b}^{i,j} \tilde{c}_{a,b})$ is a cycle in $X_{\phi}$. Thus, the eventual range, $E$, of $(f_{\phi})_*:H_1(X_{\phi})\to H_1(X_{\phi})$ is mapped onto the eventual range, $\bar{E}$, of $(f_{\phi_s})_*:H_1(X_{\phi_s})\to H_1(X_{\phi_s})$ by $(f_{\sigma})_*$.
The dual homomorphism, $((f_{\sigma})_*)^*:\bar{E}^*=Hom(\bar{E},\Q)\to E^*=Hom(E,\Q)$, is thus injective. But $((f_{\sigma})_*)^*:\bar{E}^*\to E^*$ is naturally identified with the homomorphism $\widehat{f_{\sigma}^*}:\dir f_{\phi_s}^*\to \dir f_{\phi}^*$ induced by 
$f_{\sigma}^*:H^1(X_{\phi_s})\to H^1(X_{\phi})$ (through identification of $Hom(H_1(X_{\phi_s}),\Q)$ with $H^1(X_{\phi_s})$, $\dir f_{\phi_s}^*$ with the eventual range of $f_{\phi_s}$, and the latter with $\bar{E}^*$, etc.). By means of the natural isomorphisms of $H^1(\inv f_{\phi_s})$ and $H^1(\inv f_{\phi})$ with $\dir f_{\phi_s}^*$ and $\dir f_{\phi}^*$, resp., the homomorphism $(\hat{f}_{\sigma})^*:H^1(\inv f_{\phi_s})\to H^1(\inv f_{\phi})$ induced by $\hat{f}_{\sigma}$ is identified with  $\widehat{f_{\sigma}^*}$, and hence is injective. Finally, $ \Sigma^*=\hat{\pi}^*\circ(\hat{f}_{\sigma})^*\circ (\hat{\pi}_{s}^*)^{-1}$, and since both  $\hat{\pi}^*$ and
 $\hat{\pi}_{s}^*$ are isomorphisms $\Sigma^*$ is also injective.
\end{proof}

By the {\em eventual rank} of an $m\times m$ matrix $M$ we mean the rank of the $m$-th power of $M$: evrank$(M):=\rank(M^m)$.

\begin{lemma}\label{2d} Suppose that $\iota$ is a fixed point free involution of the alphabet $\mathcal{A}$ of a substitution $\alpha$ and that $\alpha\circ\iota=\iota\circ\alpha$. If the degree of the stretching factor $\lambda$ of $\alpha$ is $d$ and the norm of $\lambda$ is odd, then the eventual rank of the abelianization of $\alpha$ is at least $2d$.
\end{lemma}
\begin{proof} The hypotheses imply that the abelianization  $M$ of $\alpha$ takes the form
$M=\begin{pmatrix}
X & Y \cr Y & X \end{pmatrix}$ where $X+Y$ is the abelianization of the substitution $\alpha_p$ induced by $\alpha$ on the alphabet of pairs $\mathcal{A}_p:=\{\{a,\iota(a)\}:a\in\mathcal{A}\}$. Since $\lambda$ is also an eigenvalue of $X+Y$ and the norm of $\lambda$ is odd, the eventual rank of $X+Y(\bmod2)$ is at least $d$. As $X-Y\equiv X+Y(\bmod2)$, the eventual rank of  $\begin{pmatrix}
X +Y& Y \cr O & X-Y \end{pmatrix}=\begin{pmatrix}
I & O \cr -I & I \end{pmatrix}\begin{pmatrix}
X & Y\cr Y & X \end{pmatrix}\begin{pmatrix}
I & O \cr I & I\end{pmatrix}$ is at least $2d$. Thus $M$ has eventual rank at least $2d$.
\end{proof}

We recall a procedure developed in \cite{BDcohomology} for computing the
cohomology of a one-dimensional substitution tiling space. Given a substitution
$\psi$, there is a pair of $1$-dimensional complexes $S\subset K$ and a
continuous $f:(K,S)\rightarrow (K,S)$ so that
$\check{H}^1(\Omega_{\Psi}) \simeq \dir( f^*:H^1(K)\rightarrow
H^1(K))$. The map $f$ induces a homomorphism of the exact sequence for
the pair $(K,S)$ yielding the commuting diagram with exact rows

\smallskip

\begin{picture}(20, 75)(-80, -35)

\put(-50, 20){$0 \arrow \tilde{H}^0(S) \stackrel{\delta}\arrow {H}^1(K, S)
\arrow {H}^1(K ) \arrow {H}^1(S ) \arrow 0$
}

\put(-50,-2){$0 \arrow \tilde{H}^0(S) \stackrel{\delta}\arrow {H}^1(K, S)
\arrow {H}^1(K ) \arrow {H}^1(S ) \arrow 0$
}

\put(-35,15){\vector(0,-1){10}}

\put(-15,15){\vector(0,-1){10}}
\put(7,15){\vector(0,-1){10}}
\put(25,15){\vector(0,-1){10}}

\put(-32, 10){$f^*_0$}
\put(-12, 10){$f^*_1$}
\put(10, 10){$f^*_2$}
\put(28, 10){$f^*_3$}

\end{picture}

The homomorphisms $f^*_i$ are induced by $f$, and, in particular,
$f^*_1$ is represented by the transpose of the abelianization,
$M_{\psi}^t$ of $\psi$.  {\parindent=0pt Taking }direct limits leads to an
exact sequence 

\begin{equation} \label{exact seq.} 0 \arrow G_0 \stackrel{\vec{\delta}}\arrow G_1 \arrow
G_2 \arrow G_3 \arrow 0
\end{equation}

in which $G_1 \simeq \dir M_{\psi}^t$,
$G_2 \simeq H^1(\Omega_{\Psi})$, and $G_0 \simeq \Q^k$, $k+1$ being
the number of connected components in the eventual range of
$f:S\rightarrow S$. An upper bound for $k+1$ is given by the number $n(\psi)$ of
bi-infinite words allowed for $\psi$ and periodic under $\psi$. Thus: $$\dim(H^1(\Omega_{\Psi}))=\dim(G_2)=\dim(G_3)+\dim(\coker(
\vec{\delta}))=\dim(G_3)+\dim(G_1)-\dim(G_0)$$ $$\ge \text{evrank}(M_{\psi}^t)-(n(\psi)-1).$$
In summary:

\begin{lemma}\label{coh bound} If the substitution $\psi$ has exactly $n(\psi)$ allowed bi-infinite words that are periodic under $\psi$, then $\dim(H^1(\Omega_{\Psi})\ge \text{evrank}(M)-n(\psi)+1$, where $M$ is the abelianization of $\psi$.
\end{lemma}

\begin{Theorem} \label{big H^1} Suppose that $\phi$ is a Pisot substitution on letters with stretching factor $\lambda$. If the degree of $\lambda$ is $d$, the norm of $\lambda$ is odd, and $cr(\Phi)=2$, then $\dim(H^1(\OP))\ge2d-1$.
\end{Theorem}
\begin{proof} 
By rewriting and passing to a power, we may assume that the substitution $\phi$ is proper. The substitution $\phi_s$ is then also proper. Let $(\phi_s)_p$ and $(\phi_s)_{op}$ be the symbolic substitutions corresponding to the tile substitutions $(\Phi_s)_p$ and $(\Phi_s)_{op}$ on $rpd$ tile pairs and $rpd$ ordered tile pairs. There is a fixed point free involution $\iota:\mathcal{A}_{op}\to\mathcal{A}_{op}$ with $\iota\circ\phi_{op}=\phi_{op}\circ\iota$. The substitution $(\phi_s)_p$ is proper but $(\phi_s)_{op}$ is definitely not. Rather, there are exactly two allowed bi-infinite words that are periodic (fixed) under $(\phi_s)_{op}$. To see that this is the case, let $T$ be the tiling, fixed by $\Phi_s$, that follows the pattern of the unique fixed (by $\phi_s$) word $\ldots b.a\ldots$ and let  $T'$ be its $rpd$ twin. Then $\{T,T,\}\in \Omega_{\Phi_s}/\sim_{rpd}\simeq\Omega_{(\Phi_s)_p}$ is the corresponding unique element of $\Omega_{(\Phi_s)_p}$ fixed by $(\Phi_s)_p$ having a tile with 0 as an end point.
If $w$ is an allowed bi-infinite word for $(\phi_s)_{op}$ that is periodic under $(\phi_s)_{op}$, let $\bar{T}_{w}\in\Omega_{(\Phi_s)_{op}}\simeq\Omega_{\Phi_s}$ be the corresponding tiling that is periodic under $(\Phi_s)_{op}$. Then $\Sigma_p(\bar{T}_w)$, being $(\Phi_s)_p$-periodic with 0 an endpoint of a tile, must equal $\{T,T'\}$. Thus,
$\bar{T}_w\in\Sigma_p^{-1}(\{T,T'\})=\{(T,T'),(T',T)\}$ (by Theorem \ref{rpd p space}, $\Sigma_p$ is exactly 2-to-1 everywhere) and we see that $w$ must be either the fixed 
bi-infinite word corresponding to $(T,T')$ or to $(T',T)$. That is, $n((\phi_s)_{op})=2$.

We have:
$$\dim(H^1(\OP))\ge \dim(H^1(\Omega_{\Phi_s}))\,\, \text{ (by Proposition \ref{injective on H^1})}$$ 
$$=\dim(H^1(\Omega_{(\Phi_s)_{op}})\,\,\text{ (by Theorem \ref{minimal model} and Lemma \ref{op})}$$ $$\ge \text{evrank}(M_{(\phi_s)_{op}})-n((\phi_s)_{op})+1\,\, \text{ (by Lemma \ref{coh bound})}$$ $$=\text{evrank}(M_{(\phi_s)_{op}})-1$$
$$\ge 2d-1,\text{ by Lemma \ref{2d}}.$$
\end{proof}

In \cite{BBJS} the substitution $\phi$ is termed {\em homological Pisot} provided the stretching factor $\lambda$ of $\phi$ is a Pisot number and the dimension of $H^1(\OP)$
equals the degree $d$ of $\lambda$,  and it is proved that if $\phi$ is homological Pisot of degree $d=1$, then $cr(\Phi)$ divides the norm of $\lambda$. Together with this, Theorem  \ref{big H^1} thus establishes the Coincidence Rank Conjecture (Conjecture \ref{CRC}) for coincidence rank 2 and arbitrary degree.

\begin{corollary} The Coincidence Rank Conjecture is true for $cr(\Phi)=2$.
\end{corollary}  

Consequently:

\begin{corollary} Any counterexample to the Homological Pisot Conjecture must have coincidence rank greater than 2.
\end{corollary} 

One-dimensional tilings $T\ne T'$ are {\em forward (backward) asymptotic} if $\lim_{t\to\infty}d(T-t,T'-t)=0$ (resp., $\lim_{t\to-\infty}d(T-t,T'-t)=0$). The tilings $T_0,T_1,\ldots,T_{2n-1}$ form an {\em asymptotic cycle} if $T_{2i}$ is forward asymptotic to $T_{2i+1}$ and $T_{2i+1}$ is backward asymptotic to $T_{2i+2}$ for $i=0,\ldots,n-1$, subscripts taken $\mod(2n)$. If $\psi$ is an irreducible Pisot substitution
and the rank of $H^1(\Omega_{\Psi})$ is greater than the degree of the stretching factor of $\psi$,
then the group $G_3$ of the exact sequence \ref{exact seq.} must be non-trivial. This can only happen if $\Omega_{\Psi}$ has an asymptotic cycle (see \cite{BD3}). We thus have:

\begin{corollary} Any counterexample to the Pisot Substitution Conjecture must have an asymptotic cycle or coincidence rank greater than two.
\end{corollary}

\section*{Acknowledgement} Thanks to Lorenzo Sadun for a simplified proof of Lemma \ref{2d}.

\end{document}